\documentclass[a4paper,10pt]{amsart}

\usepackage{graphicx}
\usepackage{geometry}
\usepackage{amsthm}
\usepackage{amssymb}
\usepackage{amsmath}
\usepackage{accents}
\usepackage{hyperref}
\usepackage{xcolor}
\usepackage{enumerate}
\usepackage{listings}
\usepackage{complexity}
\usepackage{blkarray}
\usepackage{braket}
\usepackage{lmodern}
\usepackage[english]{babel}

\usepackage[font=small,labelfont=bf]{caption}
\usepackage[font=small,labelfont=normalfont,labelformat=simple]{subcaption}
\graphicspath{{figs/}}

\newtheorem{theorem}{Theorem}

\newtheorem{lemma}[theorem]{Lemma}
\newtheorem{corollary}[theorem]{Corollary}

\newcommand{\bivec}[1]{\accentset{\leftrightarrow}{#1}}

\author
{
Raphael Steiner 
}
\thanks{Department of Computer Science, Institute of Theoretical Computer Science, ETH Z\"{u}rich, Switzerland,  \texttt{raphaelmario.steiner@inf.ethz.ch}. The author was supported by an
ETH Zurich Postdoctoral Fellowship.}

\date{\today}

\title{Subdivisions with congruence constraints in digraphs of large chromatic number}

\begin{document}
\maketitle

\begin{abstract}
We prove that for every digraph $F$ and every assignment of pairs of integers $(r_e,q_e)_{e \in A(F)}$ to its arcs there exists an integer $N$ such that every digraph $D$ with dichromatic number greater than $N$ contains a subdivision of $F$ in which $e$ is subdivided into a directed path of length congruent to $r_e$ modulo $q_e$, for every $e \in A(F)$.

This generalizes to the directed setting the analogous result by Thomassen~\cite{thomassen} for undirected graphs, and at the same time yields a novel short proof of his result.
\end{abstract}

\section{Introduction}

Given a graph $F$, a \emph{subdivision} of $F$ is a graph $H$ consisting of $v(F)$ pairwise distinct so-called \emph{branch-vertices} $(b_f)_{f \in V(F)}$ and a collection of pairwise internally disjoint so-called \emph{subdivision paths} $(P_{e})_{e \in E(F)}$, each of positive length, such that $P_e$ has endpoints $b_u$ and $b_v$ for every edge $e=uv$ in $F$. The exact same definition applies to directed graphs, with the sole difference being that for every arc $e=(u,v)$ in a digraph $F$ we require the corresponding subdivision-path $P_e$ to be directed from $b_u$ to $b_v$.

The study of subdivision containment in undirected an directed graphs has a long history in graph theory. Among the seminal results in the area is a famous theorem of Mader~\cite{mader1}, stating that for every $k \in \mathbb{N}$ there is a number $d(k)$ such that every graph of average degree at least $d(k)$ contains a subdivision of $K_k$ (and hence, of any other graph of order at most $k$). While Mader's bound on $d(k)$ was exponential, the exact order of magnitude of $d(k)$ was later determined to be $d(k)=\Theta(k^2)$ by Bollob\'{a}s and Thomason~\cite{bollobas} and independently by Koml\'{o}s and Szemer\'{e}di~\cite{komlos}. Since every graph $G$ contains a subgraph of minimum degree at least $\chi(G)-1$, the above also implies that chromatic number at least $\Theta(k^2)$ forces the containment of a $K_k$-subdivision. 

A strengthening of Hadwiger's conjecture, the so-called \emph{Haj\'{o}s conjecture}, claimed that in fact every graph of chromatic number at least $k$ contains a $K_k$-subdivision, but this is now known to be false. Explicit counterexamples for $k \ge 7$ were constructed by Catlin~\cite{catlin} and later, Erd\H{o}s and Fajtlowicz~\cite{erdosfajtlowicz} showed that there exist graphs without a $K_k$-subdivision and chromatic number at least $c\frac{k^2}{\log k}$ for some absolute constant $c>0$. 

The above results guarantee that for every fixed graph $F$, any graph $G$ of sufficiently large chromatic number will contain an $F$-subdivision. However, they give no control over the path lengths in this subdivision. In particular, the found subdivision may be a bipartite graph. In contrast to this, we know that every graph of chromatic number at least $3$ contains an odd cycle, and in fact it is even known that graphs of sufficiently large chromatic number contain cycles of length $r$ modulo $q$, for every given pair of integers $r$ and $q$, see for instance~\cite{bollobas1,thomassen,chen}. This naturally raises the question whether additional congruence constraints may be enforced on the path lengths in an $F$-subdivision in graphs of large chromatic number. Related to this, Toft conjectured in 1975 that every graph of chromatic number at least $4$ contains a \emph{totally odd subdivision} of $K_4$, that is, a subdivision with all subdivision paths of odd length. This conjecture was verified in the late 90s independently by Zang~\cite{zang} and Thomassen~\cite{thomassen2}. A qualitatively more powerful result was established by Thomassen already in 1983, as follows.

\begin{theorem}[Thomassen, Theorem~12 in~\cite{thomassen}]\label{thm:thomassen}
Let $F$ be a graph, and for every $e\in E(F)$ let $r_e, q_e \in \mathbb{Z}$ be integers with $q_e \ge 2$. 

Then there exists a number $N \in \mathbb{N}$ such that every graph $G$ with $\chi(G) \ge N$ contains a subdivision of $F$ such that for every $e \in E(F)$, the corresponding subdivision-path is of length congruent to $r_e$ modulo $q_e$.
\end{theorem}

This result largely resolves the question posed above in the positive. We refer to~\cite{alon,kawarabayashi} for some more recent results on subdivisions with congruence-constraints on the path lengths.

\medskip

In this note we prove a generalization of Theorem~\ref{thm:thomassen} to directed graphs. Accordingly, we need to specify a notion of chromatic number for directed graphs. Here, we will stick with the recently popular notion of the \emph{dichromatic number}, which was introduced by Erd\H{o}s~\cite{erdos} and Neumann-Lara~\cite{neumannlara} in 1980. Given a digraph $D$, the \emph{dichromatic number} of $D$ (denote by $\vec{\chi}(D)$) is the smallest integer $k \ge 1$ such that $V(D)$ can be partitioned into $k$ sets $V_1,\ldots,V_k$, each of them inducing an acyclic subdigraph of $D$. 

At first sight, it may not be obvious at all why sufficiently large dichromatic number forces the containment of subdivisions in directed graphs. However, using an elegant argument, this fact was recently established by Aboulker, Cohen, Havet, Lochet, Moura and Thomass\'{e}~\cite{aboulker}, see also~\cite{girao,gishboliner} for related results. 
\begin{theorem}[Aboulker et al.~\cite{aboulker}]\label{thm:aboulker}
For every digraph $F$ there exists a number $N \in \mathbb{N}$ such that every digraph $D$ with $\vec{\chi}(D) \ge N$ contains a subdivision of $F$.
\end{theorem}

Our main result stated below is a qualitative strengthening of Theorem~\ref{thm:aboulker}, additionally ensuring the path lengths in the subdivision of $F$ to satisfy given congruences.

\begin{theorem}\label{thm:main}
Let $F$ be a digraph, and for every $e\in A(F)$ let $r_e, q_e \in \mathbb{Z}$ be integers with $q_e \ge 2$. 

Then there exists a number $N \in \mathbb{N}$ such that every digraph $D$ with $\vec{\chi}(D) \ge N$ contains a subdivision of $F$ such that for every $e \in A(F)$, the corresponding subdivision-path is of length congruent to $r_e$ modulo $q_e$.
\end{theorem}

Theorem~\ref{thm:main} is the natural analogue of Thomassen's Theorem~\ref{thm:thomassen} for directed graphs. In fact, it also easily implies Thomassen's theorem, which can be seen as follows:

Let $F$ be any given graph equipped with integer pairs $(r_e,q_e)_{e \in E(F)}$. Fix an orientation $\vec{F}$ of $F$. Then by Theorem~\ref{thm:main} there exists $N \in \mathbb{N}$ such that any digraph $D$ with $\vec{\chi}(D) \ge N$ contains a subdivision of $\vec{F}$ in which every oriented edge $e$ in $\vec{F}$ is replaced by a directed path of length congruent to $r_e$ modulo $q_e$. Now, let $G$ be any given undirected graph with $\chi(G) \ge N$, and let $\bivec{G}$ denote the directed graph on the same vertex-set as $G$ with arc-set $\{(u,v),(v,u)|uv \in E(G)\}$. Pause to note that $\vec{\chi}(\bivec{G})=\chi(G) \ge N$. But then by the above $\bivec{G}$ must contain a subdivision of $\vec{F}$ with path-lengths satisfying the congruence-constraints given by the sequence $(r_e,q_e)_{e \in A(\vec{F})}$. This translates (by ignoring the directions of edges) one-to-one to a subdivision of $F$ in $G$ in which every edge $e$ is replaced by a path of length $r_e$ modulo $q_e$. 

We may conclude that Theorem~\ref{thm:thomassen} is a special case of Theorem~\ref{thm:main}. In addition, our short proof of Theorem~\ref{thm:main} is conceptually entirely different from Thomassen's proof of Theorem~\ref{thm:thomassen} in~\cite{thomassen}, which may be interesting in its own right. 

As a last point, let us mention that by considering the special case when $F$ is a directed cycle of length two, the proof of Theorem~\ref{thm:main} also implies that there exists a (reasonably small) absolute constant $C>0$ such that every digraph of dichromatic number greater than $Cq$ contains a directed cycle of length congruent to $r$ modulo $q$, for every integer $r$. This resembles the main result of Chen, Ma and Zang in~\cite{chen}, who proved this statement with the optimal constant $C=1$.

\medskip

\paragraph{\textbf{Terminology and notation.}} All graphs and directed graphs considered in this paper are simple, i.e., have no loops or parallel edges. For a graph $G$ we denote by $V(G), E(G)$ the set of vertices and edges. For a digraph $D$ we denote by $V(D)$ the vertex-set and by $A(D) \subseteq \{(u,v)\in V(D) \times V(D)|u \neq v\}$ the arc-set. We also use the notation $a(F):=|A(F)|$ for the number of arcs. For $X \subseteq V(D)$ we denote by $D[X]$ the subdigraph of $D$ consisting of the vertex-set $X$ and all arcs of $D$ going between vertices of $X$. Given a directed path $P$ or a directed cycle $C$ in $D$, we denote by $\ell(P)$ resp. $\ell(C)$ its length. Paths are allowed to consist of a single vertex (i.e., have length zero). We say that $D$ is strongly connected if for every $(x,y) \in V(D)^2$ there exists a directed path from $x$ to $y$ in $D$. A \emph{strongly connected component} of $D$ is a maximal subset $X$ of vertices such that $D[X]$ is strongly connected. For vertices $x,y$ in a strongly connected digraph $D$ we denote by $\text{dist}_D(x,y)$ the length of a shortest directed path from $x$ to $y$.

\section{Proof of Theorem~\ref{thm:main}}
The proof of Theorem~\ref{thm:aboulker} by Aboulker et al.~was based on the insight that BFS layerings in directed graphs can be used to find vertex-disjoint directed paths for a subdivision. Also our proof of Theorem~\ref{thm:main} will rely crucially on the BFS layering idea. In particular, we need Lemma~\ref{lemma:strong} and Lemma~\ref{lemma:half} stated below which both appear (in equivalent form) in~\cite{aboulker}.

\begin{lemma}[cf.~\cite{aboulker}, Lemma~29]\label{lemma:strong}
The dichromatic number of a digraph $D$ equals the maximum of the dichromatic numbers of the subdigraphs induced by its strongly connected components.
\end{lemma}

\begin{lemma}[cf.~\cite{aboulker}, Lemma~30]\label{lemma:half}
Let $D$ be a strongly connected digraph and let $v \in V(D)$. For every integer $i \ge 0$, let $L_i^+=L_i^+(v,D):=\{x \in V(D)|\text{dist}_D(v,x)=i\}$ and $L_i^-=L_i^-(v,D):=\{x \in V(D)|\text{dist}_D(x,v)=i\}$. 
Then $$\vec{\chi}(D)\le 2\max\{\vec{\chi}(D[L_i^+])|i \ge 0\}$$ and $$\vec{\chi}(D) \le 2\max\{\vec{\chi}(D[L_i^-])|i \ge 0\}.$$
\end{lemma}

We note the following simple corollary of the above.

\begin{corollary}\label{cor:connect}
Every digraph $D$ contains a set of vertices $X \subseteq V(D)$ and a vertex $x_0 \in X$ such that $D[X]$ is strongly connected, $\vec{\chi}(D[X]) \ge \frac{1}{2}\vec{\chi}(D)$ and for every $x \in X$ there is a directed path $P$ in $D$ starting at $x$, ending at $x_0$ and such that $V(P) \cap X=\{x,x_0\}$. 
\end{corollary}
\begin{proof}
By Lemma~\ref{lemma:strong} we may assume without loss of generality that $D$ is strongly connected. Pick a vertex $v$ arbitrarily and apply Lemma~\ref{lemma:half} to it. In particular, there exists an integer $i \ge 0$ such that $\vec{\chi}(D[L_i^-(v,D)]) \ge \frac{1}{2}\vec{\chi}(D)$. Let $X$ be the set of vertices of a strongly connected component of $D[L_i^-(v,D)]$ with maximum dichromatic number, i.e., $\vec{\chi}(D[X])=\vec{\chi}(D[L_i^-(v,D)]) \ge \frac{1}{2}\vec{\chi}(D)$ by Lemma~\ref{lemma:strong}. Let $P_{v,X}$ be a shortest directed path in $D$ starting at $v$ and ending at $X$. Then $P_{v,X}$ intersects $X$ only in its endpoint, call it $x_0$. 
Now, for any given $x \in X$, consider a shortest directed path $P_{x,v}$ from $x$ to $v$ in $D$ (i.e., of length exactly $i$). Since $X \subseteq L_i^-(v,D)$, we have $V(P_{x,v}) \cap X=\{x\}$.  Hence the union of the directed paths $P_{x,v}$ and $P_{v,X}$ forms a directed walk in $D$ starting at $x$, ending at $x_0$ and such that no internal vertices of it are contained in $X$. Hence, possibly after short-cutting this walk we obtain a directed $x$ to $x_0$-path $P$ in $D$ which intersects $X$ only at its endpoints. This concludes the proof.
\end{proof}

For our next step towards Theorem~\ref{thm:main}, namely Lemma~\ref{lemma:outparity} below, we need the following known statement about long cycles in digraphs of given dichromatic number, compare Corollary~38 in~\cite{aboulker} or Theorem~7 in~\cite{gishboliner}. 

\begin{lemma}\label{lemma:cyclelength}
Let $k \ge 2$ be an integer and let $D$ be a digraph. If $\vec{\chi}(D) \ge k$ then $D$ contains a directed cycle of length at least $k$.
\end{lemma}

\begin{lemma}\label{lemma:outparity}
Let $D$ be a strongly connected digraph and let $x_0 \in V(D)$, $q \in \mathbb{N}$, $q \ge 2$. Suppose that $\vec{\chi}(D) \ge 2(q-1)$. Then there exists a set $Y \subseteq V(D)$ such that $\vec{\chi}(D[Y]) \ge \frac{1}{4}\vec{\chi}(D)-\frac{q-1}{2}$ and an interval $I \subseteq \mathbb{N}$ of $q$ consecutive integers such that the following holds:

For every $y \in Y$ and every $\ell \in I$ there exists a directed path $Q$ of length $\ell$ in $D$ which starts at $x_0$, ends at $y$, and satisfies $V(Q) \cap Y=\{y\}$. 
\end{lemma}
\begin{proof}
By Lemma~\ref{lemma:half} there exists $i \ge 0$ such that $\vec{\chi}(L_i^+(D,x_0)) \ge \left\lceil\frac{1}{2}\vec{\chi}(D)\right\rceil$. Let $S$ be a strongly connected component of $D[L_i^+(D,x_0)]$ with $\vec{\chi}(D[S])=\vec{\chi}(L_i^+(D,x_0)) \ge \left\lceil\frac{1}{2}\vec{\chi}(D)\right\rceil$. 

Further, let $Z \subseteq S$ be chosen inclusion-wise minimal such that $\vec{\chi}(D[Z]) \ge \left\lceil\frac{1}{2}\vec{\chi}(D)\right\rceil-(q-1)$. Then the minimality of $Z$ and Lemma~\ref{lemma:strong} imply that $\vec{\chi}(D[Z])=\left\lceil \frac{1}{2}\vec{\chi}(D)\right\rceil-(q-1)$ and that $D[Z]$ is strongly connected. It follows that $\vec{\chi}(D[S \setminus Z]) \ge \vec{\chi}(D[S])-\vec{\chi}(D[Z]) \ge q-1$. 

We claim that there exists a directed path $P$ in $D[S]$ of length exactly $q-1$ such that $P$ intersects $Z$ only at its endpoint. This holds trivially if $q=1$. If $q=2$ we have $\vec{\chi}(D[S \setminus Z]) \ge 1$ and thus $S \setminus Z \neq \emptyset$. Hence, the strong connectivity of $D[S]$ implies the existence of an arc from $S \setminus Z$ to $Z$ in $D$. Next, suppose $q \ge 3$. Then we may apply Lemma~\ref{lemma:cyclelength} to $D[S \setminus Y]$ with $k=q-1$ and find a directed cycle $C$ in $D[S \setminus Y]$ of length at least $q-1$. Let $R$ be a shortest directed path in $D[S]$ starting in $V(C)$ and ending in $Z$. Let $r_1 \in V(C)$ and $r_2 \in Z$ denote the start- and endpoint of $R$, such that $V(R) \cap V(C)=\{r_1\}$, $V(R) \cap Z=\{r_2\}$. We may now consider the directed path in $D[S]$ obtained by concatenating the segment of $C$ of length $q-2$ ending at $r_1$ with $R$. This path has length at least $q-1$ and intersects $Z$ only at its endpoint $r_2$. Hence we may find a subpath $P$ of $R$ of length exactly $q-1$ ending at $r_2$, as desired.

Moving on, consider the vertex $r_2$ in the strongly connected digraph $D[Z]$ and apply Lemma~\ref{lemma:half} to it. We then find an integer $j \ge 0$ such that $\vec{\chi}(D[L_j^+(r_2,D[Z])]) \ge \frac{1}{2}\vec{\chi}(D[Z]) \ge \frac{1}{4}\vec{\chi}(D)-\frac{q-1}{2}$. We now define $Y:=L_j^+(r_2,D[Z])$ and $I:=[i+j,i+j+(q-1)]$. To conclude the proof, let a vertex $y \in Y$ and a number $\ell\in I$ be given arbitrarily. Write $l=i+j+\alpha$ for some $0 \le \alpha \le q-1$, and let $p(\alpha) \in V(P)$ denote the vertex at distance exactly $\alpha$ from $r_2$ along $P$.  Let $P_1$ be a shortest path from $x_0$ to $p(\alpha)$ in $D$. Since $p(\alpha) \in S \subseteq L_i^+(D,x_0)$, we have $\ell(P_1)=i$. Further, let $P_2$ be a shortest path from $r_2$ to $y$ in $D[Z]$. Then $\ell(P_2)=j$ since $y \in L_j^+(r_2,D[Z])$. Hence the path $Q$ formed as the union of $P_1$, the segment of $P$ from $p(\alpha)$ to $r_2$, and $P_2$ is a directed path of length $i+\alpha+j=\ell$ in $D$ from $x_0$ to $y$. Since $P_1, P_2$ are shortest paths we further we have $V(P_1) \cap S=\{p(\alpha)\}, V(P) \cap Z \subseteq V(R) \cap Z=\{r_2\}, V(P_2) \cap Y=\{y\}$. This implies that $V(Q) \cap Y=\{y\}$, as desired. This concludes the proof.
\end{proof}

We now give the proof of Theorem~\ref{thm:main}. The idea is to combine Corollary~\ref{cor:connect} and Lemma~\ref{lemma:outparity} to enable induction on the number of arcs in $F$. 

\begin{proof}[Proof of Theorem~\ref{thm:main}]
We prove the statement by induction of the number of arcs $a(F)$ in $F$. If $a(F)=0$, then the claim holds true trivially (every digraph of dichromatic number $N:=v(F)$ has at least $v(F)$ vertices and thus contains a copy of $F$). 

Moving on, suppose that $a(F) \ge 1$ and that we have established Theorem~\ref{thm:main} for all digraphs containing strictly fewer arcs. Let  $(r_e,q_e)_{e \in A(F)}$ be a given assignment of integer pairs to the arcs of $F$. We pick an arc $f \in A(F)$ arbitrarily. Let $N' \in \mathbb{N}$ be a number such that every digraph of dichromatic number at least $N'$ contains a subdivision of $F-f$ in which for every arc $e \in A(F)\setminus\{f\}$ the corresponding subdivision path is of length congruent to $r_e$ modulo $q_e$. 

Define $N:=8N'+4(q_f-1)$ and let $D$ be any given digraph with $\vec{\chi}(D) \ge N$. Let a subset $X$ of $V(D)$ and a vertex $x_0 \in X$ be as guaranteed by Corollary~\ref{cor:connect}, applied to $D$. Then $\vec{\chi}(D[X]) \ge \frac{1}{2}\vec{\chi}(D)$, $D[X]$ is strongly connected and every vertex $x \in X$ can be connected to $x_0$ via a directed path intersecting $X$ only at its endpoints. Since $D[X]$ is strongly connected, we may apply Lemma~\ref{lemma:outparity}. We thus find $Y \subseteq X$ with $\vec{\chi}(D[Y]) \ge \frac{1}{4}\vec{\chi}(D[X])-\frac{q_f-1}{2} \ge \frac{1}{8}\vec{\chi}(D)-\frac{q_f-1}{2} \ge N'$ and an interval $I$ of $q_f$ consecutive integers such that the following holds: For every $y \in Y$ and every $\ell \in I$ there is a directed $x_0$ to $y$-path of length $\ell$ in $D[X]$ intersecting $Y$ only in $y$. By inductive assumption, there exists a subdivision of $F-f$ in $D[Y]$ such that every $e \in A(F)\setminus\{f\}$ is replaced by a path of length congruent to $r_e$ modulo $q_e$. Let $y_1, y_2 \in Y$ be the two branch-vertices in this subdivision corresponding to the start- respectively endpoint of the arc $f$ in $F$. We will now show that there exists a directed path $P$ from $y_1$ to $y_2$ in $D$ intersecting $Y$ only at its endpoints, with $\ell(P) \equiv_{q_f} r_f$. Adding this path to the subdivision of $F-f$ then yields a subdivision of $F$ in $D$ with the desired congruence properties. 

To construct $P$, we first connect $y_1$ to $x_0$ via a directed path $P_0$ with $V(P_0) \cap X=\{y_1,x_0\}$ (in particular, $P_0$ is internally disjoint from $Y \subseteq X$). Let $\ell \in I$ be a number chosen such that $\ell \equiv_{q_f} r_f-\ell(P_0)$ (this is possible since $I$ is an interval with $q_f$ elements). Then by our choice of $Y$ there exists a path $Q$ in $D[X]$ from $x_0$ to $y_2$ such that $\ell(Q)=\ell$ and $Q$ is internally disjoint from $Y$. Overall, $P_0 \cup Q$ is a directed $y_1$-$y_2$-path in $D$ which is internally disjoint from $Y$ and satisfies $\ell(P_0 \cup Q)=\ell(P_0)+\ell \equiv_{q_f} \ell(P_0)+(r_f-\ell(P_0))=r_f$. Thus, by adding this path to the subdivision of $F-f$ in $D[Y]$ we find a subdivision of $F$ in $D$ in which every arc $e$ is replaced by a path of length congruent to $r_e$ modulo $q_e$. This proves the inductive claim and concludes the proof of the theorem.
\end{proof}

\end{document}